\documentclass{amsart}

\usepackage{bm}
\usepackage{psfrag}
\usepackage{color}
\usepackage{enumerate}
\usepackage{algpseudocode}

\usepackage{xy}
\xyoption{all}

\usepackage[utf8]{inputenc}
\usepackage[hidelinks]{hyperref} 
\usepackage{graphicx} 
\usepackage{amssymb,amsmath,amsthm} 
\usepackage{mathtools} 
\usepackage{enumitem} 
\usepackage[english]{babel} 
\usepackage{framed}
\usepackage{mdframed}
\usepackage[dvipsnames]{xcolor}
\usepackage{tikz}
\usepackage{listings}
\usepackage{url}
\newcommand{\floor}[1]{\left\lfloor #1 \right\rfloor}

\newcommand{\cone}{\operatorname{cone}}

\newcommand{\R}{\mathbb{R}}
\newcommand{\Z}{\mathbb{Z}}

\newcommand{\setcond}[2]{\left\{ #1 \,:\, #2 \right\}}

\newcommand{\vol}{\operatorname{vol}}

\newcommand{\gen}[1]{\left< #1 \right>}

\newcommand{\excise}[1]{}

\newcommand{\ve}{\boldsymbol}

\newtheorem{thm}{Theorem}
\newtheorem{lemma}[thm]{Lemma}
\newtheorem{cor}[thm]{Corollary}

\newtheorem{prob}[thm]{Problem}

\theoremstyle{definition}

{\bfseries}{\normalfont}
\newtheorem{remark}[thm]{Remark}{\bfseries}{\normalfont}

{\bfseries}{\normalfont}

\newcommand{\be}{\begin{eqnarray}}
\newcommand{\bea}{\begin{eqnarray*}}
\newcommand{\ee}{\end{eqnarray}}
\newcommand{\eea}{\end{eqnarray*}}

\newcommand{\ring}[1]{\ensuremath{\mathbb{#1}}}

\newcommand{\lattice}{\mathcal{L}}

\newcommand{\ICR}{\mathsf{ICR}}

\newcommand\ZZ{\ring{Z}}

\newcommand\icone{{\mathcal Sg}}

\def\udots{\mathinner{\mkern1mu\raise1pt\vbox{\kern7pt\hbox{.}}\mkern2mu
	\raise4pt\hbox{.}\mkern2mu\raise7pt\hbox{.}\mkern1mu}}

\title{Optimizing Sparsity over Lattices and Semigroups}

\author{Iskander Aliev, Gennadiy Averkov, Jes\'us A. De Loera, and Timm Oertel}

\date{\today }

\begin{document}
\maketitle

\begin{abstract}
	\noindent
	Motivated by problems in optimization we study the \emph{sparsity} of the solutions to systems of linear Diophantine 
	equations and linear integer programs, 
	i.e., the number of non-zero entries of a solution, which is often referred to as the $\ell_0$-norm.
	Our main results are improved bounds on the $\ell_0$-norm of sparse solutions to systems $A\ve x=\ve b$, where $A\in \Z^{m\times n}$, ${\ve b}\in \Z^m$ and $\ve x$ 
	is either a general integer vector (lattice case) or a non-negative integer vector (semigroup case). In the lattice case and certain scenarios of the semigroup case, we give polynomial time algorithms for computing solutions with $\ell_0$-norm satisfying the obtained bounds.
\end{abstract}


\section{Introduction}

This paper discusses the problem of finding sparse solutions to systems of linear Diophantine equations 
and integer linear programs. 
We investigate the $\ell_0$-norm $\|\ve x\|_0 := |\setcond{ i }{x_i \ne 0}|$, a function widely used in the theory of {\em compressed sensing} \cite{CS_survey,candestao}, which measures the sparsity of a given vector $\ve x = (x_1,\ldots,x_n)^\top \in \R^n$ (it is clear that the $\ell_0$-norm is actually  not a norm).

Sparsity is a topic of interest in several areas of optimization.
The $\ell_0$-norm minimization problem over reals is central in the theory of the classical compressed sensing, 
where a linear programming relaxation provides a guaranteed approximation \cite{candesetal2006stable,candestao}.
Support minimization for solutions to Diophantine equations is relevant for the theory of compressed sensing for discrete-valued
signals \cite{PROMP,Lenya,Konyagin}. There is still little understanding of discrete signals in the compressed sensing paradigm, despite the fact that there are many applications in which the signal is known to have discrete-valued entries, for instance, in wireless communication \cite{MIMO} and the theory of error-correcting codes \cite{ECC}.  
Sparsity was also investigated in integer optimization \cite{Support,EisenbrandShmonin2006,MR3950898}, where many combinatorial optimization problems have useful interpretations as  sparse semigroup problems. For example, the edge-coloring problem can be seen as a problem in the semigroup generated by matchings of the graph \cite{LOVASZ1987187}. Our results provide natural out-of-the-box sparsity bounds for problems with linear constraints and integer variables in a general form.

\subsection{Lattices: sparse solutions of linear Diophantine systems}

Each integer matrix $A \in \Z^{m \times n}$ determines the lattice 
$\lattice(A):= \setcond{A \ve x }{\ve x \in \Z^n}$
generated by the columns of $A$. By an easy reduction via row transformations, we may assume without loss of generality that the rank of $A$ is $m$. 

Let $[n]:= \{1,\ldots,n\}$ and let $\binom{[n]}{m}$ be the set of all $m$-element subsets of $[n]$. For $\gamma \subseteq [n]$, consider the $m \times |\gamma|$ submatrix $A_\gamma$ of $A$ with columns indexed by $\gamma$. One can easily prove that the determinant of $\lattice(A)$ is equal to
\[
\gcd(A):= \gcd \setcond{\det(A_\gamma)}{\gamma \in \binom{[n]}{m}}.
\]

Since $\lattice(A_\gamma)$ is the lattice spanned by the columns of $A$ indexed by $\gamma$, it is a sublattice of $\lattice(A)$. We first deal with a natural question: {\em Can the description of a given lattice $\lattice(A)$ in terms of $A$ be made sparser by passing from $A$ to $A_\gamma$ with $\gamma$ having a smaller cardinality than $n$ and satisfying $\lattice(A) = \lattice(A_\gamma)$? } That is, we want to discard some of the columns of $A$ and generate $\lattice(A)$ by $|\gamma|$ columns with $|\gamma|$ being possibly small.



For stating our results, we need several number-theoretic functions. Given $z \in \Z_{>0}$, consider the prime factorization $z = p_1^{s_1} \cdots p_k^{s_k}$ with pairwise distinct prime factors $p_1,\ldots,p_k$ and their multiplicities $s_1,\ldots,s_k \in \Z_{>0}$. 
Then the number of prime factors $\sum_{i=1}^k s_i$  counting the multiplicities is denoted by $\Omega(z)$. Furthermore, we 
introduce $\Omega_m(z) := \sum_{i=1}^k \min \{s_i,m\}$. That is, by introducing $m$ we set a threshold to account for multiplicities. In the case $m=1$ we thus have $\omega(z):=\Omega_1(z) = k$, which is the number of prime factors in $z$, not taking the multiplicities into account. The functions $\Omega$ and $\omega$ are called \emph{prime $\Omega$-function} and \emph{prime $\omega$-function}, respectively, in number theory \cite{hardy2008introduction}. We call $\Omega_m$ the \emph{truncated prime $\Omega$-function}.

\begin{thm}
	\label{thm:sparsifying:lattice:A}
	Let $A \in \Z^{m \times n}$, with $m \le n$, and let $\tau \in \binom{[n]}{m}$ be such that the matrix $A_\tau$ is non-singular. Then the equality $\lattice(A) = \lattice(A_\gamma)$
	holds for some $\gamma$ satisfying $\tau \subseteq \gamma \subseteq [n]$ and
	\begin{equation}
		\label{card:gamma:bound}
		|\gamma| \le m + \Omega_m \left( \frac{|\det(A_\tau)|}{\gcd(A)} \right).
	\end{equation}
	Given $A$ and $\tau$, the set $\gamma$ can be computed in polynomial time.
\end{thm}

One can easily see that $\omega(z) \le \Omega_m(z) \le \Omega(z) \le \log_2 (z)$ for every $z  \in \Z_{>0}$. The  estimate using $\log_2(z)$ gives a first impression on the quality of the bound \eqref{card:gamma:bound}. It turns out, however, that $\Omega_m(z)$ is much smaller on the average. Results in number theory \cite[\S22.10]{hardy2008introduction} show that the average values $\frac{1}{z}(\omega(1) + \cdots + \omega(z))$ and $\frac{1}{z}(\Omega(1) + \cdots + \Omega(z))$ are of order $\log \log z$, as $z \to \infty$.

As an immediate consequence of Theorem~\ref{thm:sparsifying:lattice:A} we obtain

\begin{cor}\label{cor:sparsity:dioph:sys}
Consider the linear Diophantine system
\begin{equation}\label{lin:dioph:sys}
	A \ve x = \ve b, \ \ve x \in \Z^n
\end{equation}
with $A \in \Z^{m \times n}$, $\ve b \in \Z^{m }$ and $m \le n$.  Let $\tau \in \binom{[n]}{m}$ be such that the $m \times m$ matrix $A_\tau$
is non-singular.  If \eqref{lin:dioph:sys} is feasible, then \eqref{lin:dioph:sys} has a solution $\ve x$ satisfying the sparsity bound

\begin{equation*}
	\|\ve x \|_0 \le m + \Omega_m \left( \frac{| \det(A_\tau)| }{\gcd(A)} \right).
\end{equation*}
Under the above assumptions, for given $A,\ve b$ and $\tau$, such a sparse solution can be computed in polynomial time.
\end{cor}

From the optimization perspective, Corollary~\ref{cor:sparsity:dioph:sys} deals with the problem
\begin{equation*}
	\min \setcond{ \|\ve x\|_0 }{ A \ve x = \ve b, \ \ve x \in \Z^n}
\end{equation*}
of minimization of the $\ell_0$-norm over the affine lattice $\setcond{\ve x \in \Z^n}{A \ve x = \ve b}$.

\subsection{Semigroups: sparse solutions in integer programming}

Consider next the standard form of the feasibility constraints of integer linear programming
\begin{equation}\label{ILP:feasibility}
	A \ve x =\ve b, \ \ve x \in \Z_{\ge 0}^n.
\end{equation}

For a given matrix $A$, the set of all $\ve b$ such that \eqref{ILP:feasibility} is feasible, is the \emph{ semigroup}
$ \icone(A) = \{ A \ve x : \ve x \in \mathbb{Z}_{\ge0}^n \}$ generated by the columns of $A$. 
 
If \eqref{ILP:feasibility} has a solution, i.e., $\ve b\in\icone(A)$, \emph{ how sparse can such a solution be?} In other words, we are interested in the $\ell_0$-norm minimization problem
\begin{equation}\label{l0:IP}
	\min \setcond{ \|\ve x\|_0 }{A \ve x = \ve b, \ \ve x \in \Z_{\ge 0}^n}.
\end{equation}
It is clear that Problem \eqref{l0:IP} is NP-hard, because deciding the feasibility of \eqref{ILP:feasibility} \cite[\S~18.2]{Schrijver} or even solving the relaxation of  \eqref{l0:IP} with the condition ${\ve x} \in \Z_{\ge 0}^n$ replaced by ${\ve x} \in \R^n$  \cite{natarajan} is NP-hard. 

Taking the NP-hardness of Problem (\ref{l0:IP}) into account, our aim is to \emph{estimate} the optimal value of \eqref{l0:IP} under the assumption that this problem 
is feasible.  In \cite[Theorem 1.1 (i)]{ADOO} (see also \cite[Theorem 1]{Support}), it was shown that for any $\ve b \in \icone(A)$, 
there exists a $\ve x \in \Z^n$, such that $A\ve x = \ve b$ and 
\begin{equation}\label{ADNO:bound}
	\|\ve x\|_0 \le m + \floor{ \log_2 \left( \frac{\sqrt{\det(A A^\top)}}{\gcd(A)} \right) }.
\end{equation}

We show that in some special cases we can significantly improve this bound. As a consequence of Theorem~\ref{thm:sparsifying:lattice:A} we obtain the following. 
\begin{cor}
	\label{cor:pos:spanning:columns}
Let $A \in \Z^{m \times n}$ be a matrix whose columns positively span $\R^m$ and let $\ve b \in \Z^m$. Then $ \lattice(A)=\icone(A)$. Furthermore, if $\ve b \in \lattice(A)$, and  $\tau \in \binom{[n]}{m}$ is a set, for which the matrix $A_\tau$ is non-singular, then there is a solution $\ve x$ of the integer-programming feasibility problem $A \ve x  = \ve b, \ve x \in \Z_{\ge 0}^m$ that satisfies the sparsity bound
	\begin{equation}
		\label{pos:spanning:bound}
		\|\ve x\|_0 \le 2m + \Omega_m \left( \frac{|\det(A_\tau)|}{\gcd(A)} \right).
	\end{equation}
	Under the above assumptions, for given $A, \ve b$ and $\tau$, such a sparse solution $\ve x$ can be computed in polynomial time.
\end{cor}

Note that for a fixed $m$, \eqref{pos:spanning:bound} is usually much tighter than \eqref{ADNO:bound}, because the function $\Omega_m(z)$ is bounded from above by the logarithmic function $\log_2(z)$ and is much smaller than $\log_2(z)$ on the average. Furthermore, $|\det(A_\tau)| \le \sqrt{\det(A A^\top)}$ in view of the Cauchy-Binet formula.

We take a closer look at the case $m=1$ of a single equation and tighten the given bounds in this case. 
That is, we consider the \emph{knapsack feasibility problem} 
\begin{equation}\label{knapsack_problem}
	\ve a^\top \ve x = b, \ \ve x \in \Z_{\ge 0}^n,
\end{equation}
where $\ve a \in \Z^n$ and $b \in \Z$. Without loss of generality we can assume that all components of the vector $\ve a$ are not equal to zero. It follows from (\ref{ADNO:bound}) that a feasible problem (\ref{knapsack_problem}) has a solution ${\ve x}$ with
\be\label{previous_knapsack_bound}
\|\ve x\|_0 \le 1 + \left\lfloor\log\left (\frac{\|{\ve a}\|_2}{\gcd({ \ve a})}\right )\right\rfloor\,.
\ee
If all components of $\vec a$ have the same sign, without loss of generality we can assume $\vec a \in \Z_{>0}^n$. In this setting, Theorem 1.2 in \cite{ADOO} strengthens the bound (\ref{previous_knapsack_bound}) by replacing the $\ell_2$-norm of the vector $\ve a$ with the $\ell_{\infty}$-norm. It was  conjectured in \cite[page~247]{ADOO} 
that a bound $\|\ve x\|_0 \le c + \left\lfloor\log_2\left (\|{ \ve a}\|_{\infty}/\gcd({ \ve a})\right )\right\rfloor\,$ with an absolute  constant $c$ holds for an \emph{arbitrary} $\ve a \in \Z^n$. We obtain the following result, which covers the case that has not been settled so far and yields a confirmation of this conjecture.

\begin{cor}
	\label{cor:mixed:sign:knapsack}
	Let $\ve a = (a_1,\ldots,a_n)^\top \in (\Z \setminus \{0\})^n$ be a vector that contains both positive and negative components. If the knapsack feasibility problem $\ve a^\top \ve x = b, \ \ve x \in \Z_{\ge 0}^n$ has a solution, then there is a solution $\ve x$ satisfying the sparsity bound
	\[
		\|\ve x\|_0 \le 2 + \min \setcond{ \omega \left( \frac{|a_i|}{\gcd(\ve a)}\right) }{i \in [n]}.
	\]
	Under the above assumptions, for given $\ve a$ and $b$, such a sparse solution $\ve x$ can be computed in polynomial time.
\end{cor}

Our next contribution is that, given additional structure on $A$, we can improve on \cite[Theorem 1.1 (i)]{ADOO}, which in turn also gives an improvement on \cite[Theorem 1.2]{ADOO}.
For $\ve a_1,\ldots,\ve a_n \in \R^m$, we denote by $\cone(\ve a_1,\ldots,\ve a_n)$ the convex conic hull of the set $\{\ve a_1,\ldots,\ve a_n\}$. 
Now assume the matrix $A = (\ve a_1,\ldots,\ve a_n) \in \Z^{m \times n}$ with columns ${\ve a}_i$ satisfies the following conditions: 
\begin{align}
& \ve a_1,\ldots,\ve a_n \in \Z^m \setminus \{\ve 0\}, \label{non_zero_col}
\\ & \text{$\cone(\ve a_1,\ldots,\ve a_n)$ is an $m$-dimensional pointed cone}, \label{m_dim_pointed_cone}
\\ & \text{$\cone(\ve a_1)$ is an extreme ray of $\cone(\ve a_1,\ldots,\ve a_n)$}. \label{extreme_ray}
\end{align}
Note that the previously best sparsity bound for the general case of the integer-programming feasibility problem is \eqref{ADNO:bound}.
Using the Cauchy-Binet formula, \eqref{ADNO:bound} can be written as
\[
	\|\ve x\|_0 \le m + \log_2 \frac{\sqrt{\sum_{I \in \binom{[n]}{m}}\det(A_I)^2}}{\gcd(A)}.
\]
The following theorem improves this bound in the \emph{``pointed cone case''} by removing a fraction of $m/n$ of terms in the sum under the square root. 

\begin{thm}\label{thm:pointed:cone}
	Let $A = (\ve a_1,\ldots,\ve a_n) \in \Z^{m \times n}$ satisfy \eqref{non_zero_col}--\eqref{extreme_ray} and, for $\ve b \in \Z^m$, consider the integer-programming feasibility problem 
	\begin{equation} \label{IP_feas}
		A \ve x =\ve b, \ \ve x \in \Z_{\ge 0}^n.
	\end{equation}
	If  \eqref{IP_feas} is feasible, then there is a feasible solution $\ve x$ satisfying the sparsity bound 
	\[
		\|\ve x\|_0 \le m + \floor{\log_2 \frac{q(A)}{\gcd(A)}},
	\] 
	where
 	\begin{equation*} 
		q(A):= \sqrt{\sum_{I \in \binom{[n]}{m} \,:\, 1 \in I} \det(A_I)^2}.
	\end{equation*}
\end{thm}

We omit the proof of this result due to the page limit for the IPCO proceedings.
Instead we focus on the particularly interesting case  $m=1$. In this case,  assumption~\eqref{m_dim_pointed_cone} is equivalent to $\ve a \in \Z_{>0}^n \cup \Z_{<0}^n$. Without loss of generality, one can assume $\ve a \in \Z_{>0}^n$.

\begin{thm}
	\label{thm:positive:knapsack}
	Let $\ve a = (a_1,\ldots,a_n)^\top \in \Z_{>0}^n$ and $b \in \Z_{\ge 0}$. If the knapsack feasibility problem $\ve a^\top \ve x = b, \  \ve x \in \Z_{\ge 0}^n$ has a solution, there is a solution $\ve x$ satisfying the sparsity bound
	\begin{equation*}
		\|\ve x\|_0 \le 1 + \floor{\log_2 \left( \frac{\min \{a_1,\ldots,a_n\}}{\gcd(\ve a)} \right)}.
	\end{equation*}
\end{thm}

\noindent When dealing with bounds for sparsity it would be interesting to understand \emph{the worst case scenario among all members of the semigroup}, which is described by the function
\begin{equation}\label{def:ICR}
	\ICR(A) = \max_{ \ve b \in \icone(A)} \min\{ \| {\ve x} \|_0 \,:\,  A \ve x = \ve b, \ {\ve x} \in \ZZ_{\ge 0}^n\}.
\end{equation}
We call $\ICR(A)$ the \emph{integer Carath\'eodory rank} in resemblance to the classical problem of finding the integer Carath\'eodory number for Hilbert bases \cite{SeboCaratheodory}. Above results for the problem $A \ve x =b, \ \ve x \in \Z_{\ge 0}^n$ can be phrased as upper bounds on $\ICR(A)$. We are interested in the complexity of computing $\ICR(A)$.
 The first question is: \emph{can the integer Carath\'eodory rank of a matrix $A$ be computed at all?} After all, remember that the semigroup has infinitely many elements and, despite the fact that $\ICR(A)$ is a finite number, a direct usage of \eqref{def:ICR} would result into the determination of the sparsest representation $A \ve x = \ve b$ for all of the infinitely many elements $\ve b$ of $\icone(A)$. It turns out that $\ICR(A)$ is computable, as the inequality $\ICR(A) \le k$ can be expressed as the formula $\forall \ve x \in \Z_{\ge 0}^n \, \exists \ve y \in \Z_{\ge 0}^n \,: \, (A \ve x = A \ve y) \wedge (\|\ve y\|_0 \le k )$ in \emph{Presburger arithmetic} \cite{haase2018survival}.  Beyond this fact, the complexity status of computing $\ICR(A)$ is largely open, even when $A$ is just one row:

\begin{prob} \label{prob:ICR}
	Given the input $\ve a = (a_1,\ldots,a_n)^\top \in \Z^n$, is it NP-hard to compute $\ICR(\ve a^\top)$?
\end{prob}

The \emph{Frobenius number} $\max\,  \Z_{\ge 0} \setminus \icone(\ve a^\top)$, defined under the assumptions $\ve a \in \Z_{>0}^n$ and $\gcd({\ve a})=1$, is yet another value associated to 
$\icone(\ve a^\top)$. The Frobenius number can be computed in polynomial time when $n$ is fixed \cite{barvinok-woods-2003,kannan1992lattice} but is NP-hard to compute
when $n$ is not fixed \cite{ramirez1996complexity}. It seems that there might be a connection between computing the Frobenius number and $\ICR(\ve a^\top)$.

\section{Proofs of Theorem~\ref{thm:sparsifying:lattice:A} and its consequences}

The proof Theorem~\ref{thm:sparsifying:lattice:A} relies on the theory of finite Abelian groups. 	
We write Abelian groups additively. An Abelian group $G$ is said to be a \emph{direct sum} of its finitely many subgroups $G_1,\ldots,G_m$, which is written as $G = \bigoplus_{i=1}^m G_i$, if every element $x \in G$ has a unique representation as $x = x_1 + \cdots + x_m$ with $x_i \in G_i$ for each $i \in [m]$. A \emph{primary cyclic group} is a non-zero finite cyclic group whose order is a power of a prime number. 
We use $G / H$ to denote the quotient of $G$ modulo its subgroup $H$.

	The fundamental theorem of finite Abelian groups states that every finite Abelian group $G$ has a \emph{primary decomposition}, which is essentially unique. This means, $G$ is decomposable into a direct sum of its primary cyclic groups and that this decomposition is unique up to automorphisms of $G$. We denote by $\kappa(G)$ the number of direct summands in the primary decomposition of $G$.  
	
For a subset $S$ of a finite Abelian group $G$, we denote by $\gen{S}$ the subgroup of $G$ generated by $S$. We call a subset $S$ of $G$ \emph{non-redundant} if the subgroups $\gen{T}$ generated by proper subsets $T$ of $S$ are properly contained in $\gen{S}$. In other words, $S$ is non-redundant if $\gen{S \setminus \{x\}}$ is a proper subgroup of $\gen{S}$ for every $x \in S$.  The following result can be found in  \cite[Lemma~A.6]{Geroldinger+HalterKoch}.
\begin{thm}
	\label{thm:max:card:nonredundant}
	Let $G$ be a finite Abelian group. Then
	the maximum cardinality of a non-redundant subset $S$ of $G$ is equal to $\kappa(G)$. 
\end{thm}

We will also need the following lemmas. 
	
	\begin{lemma}
		\label{lem:kappa:for:sum:cyclic}
		Let $G$ be a finite Abelian group representable as a direct sum $G = \bigoplus_{j=1}^m G_j$ of $m \in \Z_{>0}$ cyclic groups. Then $\kappa(G) \le \Omega_m(|G|)$.
	\end{lemma}
\begin{proof}
		Consider the prime factorization $|G| = p_1^{n_1} \cdots p_s^{n_s}$. Then $|G_j| = p_1^{n_{i,j}} \cdots p_s^{n_{i,j}}$ with $0 \le n_{i,j} \le n_i$ and, by the Chinese Remainder Theorem, the cyclic group $G_j$ can be represented as $G_j = \bigoplus_{i=1}^s G_{i,j}$, where $G_{i,j}$ is a cyclic group of order $p_i^{n_{i,j}}$. Consequently, $G = \bigoplus_{i=1}^s \bigoplus_{j=1}^m G_{i,j}$.
		This is a decomposition of $G$ into a direct sum of primary cyclic groups and, possibly, some  trivial summands $G_{i,j}$ equal to $\{0\}$. We can count the non-trivial direct summands whose order is a power of $p_i$, for a given $i \in [s]$. There is at most one summand like this for each of the groups $G_j$. So, there are at most $m$ non-trivial summands in the decomposition whose order is a power of $p_i$. On the other hand, the direct sum of all non-trivial summands whose order is a power of $p_i$ is a group of order $p_i^{n_{i,1} + \cdots + n_{i,s}} = p_i^{n_i}$ so that the total number of such summands is not larger than $n_i$, as every summand contributes the factor at least $p_i$ to the power $p_i^{n_i}$. This shows that the total number of non-zero summands in the decomposition of $G$ is at most $\sum_{i=1}^s \min \{m, n_i\} = \Omega_m(|G|)$.
\end{proof}	
	

	\begin{lemma}
		\label{lem:group:from:lattice}
		Let $\Lambda$ be a sublattice of $\Z^m$ of rank $m \in \Z_{>0}^m$. Then $G = \Z^m / \Lambda$ is a finite Abelian group of order $\det(\Lambda)$ that can be represented as a direct sum of at most $m$ cyclic groups.
	\end{lemma}
	
\begin{proof}
		The proof relies on the relationship of finite Abelian groups and lattices, see \cite[\S4.4]{Schrijver}. 
		Fix a matrix $M \in \Z^{m \times m}$ whose columns form a basis of $\Lambda$. Then $|\det(M)| = \det(\Lambda)$. There exist unimodular matrices $U \in \Z^{m \times m}$ and $V \in \Z^{m \times m}$ such that $D:= U M V$ is diagonal matrix with positive integer diagonal entries. For example, one can choose $D$ to be the Smith Normal Form of $M$  \cite[\S4.4]{Schrijver}. Let $d_1,\ldots,d_m \in \Z_{>0}$ be the diagonal entries of $D$.   Since $U$ and $V$ are unimodular, $d_1 \cdots d_m = \det(D) = \det(\Lambda)$. 
		
		We introduce the quotient group $G' := \Z^m / \Lambda' = ( \Z / d_1 \Z) \times \cdots \times (\Z / d_m \Z)$
		with respect to the lattice $\Lambda' := \lattice(D) = (d_1 \Z) \times \cdots \times (d_m \Z)$.
		The order of $G'$ is $d_1 \cdots d_m = \det(D) = \det(\Lambda)$ and $G'$ is a direct sum of at most $m$ cyclic groups, as every $d_i > 1$ determines a non-trivial direct summand. 
		
		To conclude the proof, it suffices to show that $G'$ is isomorphic to $G$. To see this, note that $\Lambda' = \lattice(D) = \lattice(UMV) = \lattice(UM) = \setcond{U z}{z \in \Lambda}$.
		Thus, the map $z \mapsto U z$ is an automorphism of  $\Z^m$ and an isomorphism from  $\Lambda$ to $\Lambda'$. Thus, $z \mapsto U z$ induces an isomorphism from the group $G = \Z^m / \Lambda$ to the group $G'=\Z^m / \Lambda'$. 
	\end{proof}	
%
%
%

	\begin{proof}[Proof of Theorem~\ref{thm:sparsifying:lattice:A}]
			Let $\vec a_1,\ldots,\vec a_n$ be the columns of $A$. Without loss of generality, let $\tau = [m]$. We use the notation $B:=A_\tau$.
			
			\emph{Reduction to the case $\gcd(A)=1$.} For a non-singular square matrix $M$, the columns of $M^{-1} A$ are representations of the columns of $A$ in the basis of columns of $M$. In particular, for a matrix $M$ whose columns form a basis of $\lattice(A)$, the matrix $M^{-1} A$ is integral and the $m \times m$ minors of $M^{-1} A$ are the respective $m \times m$ minors of $A$ divided by $\det(M) = \gcd(A)$. Thus, replacing $A$ by $M^{-1} A$, we pass from $\lattice(A)$ to $\lattice(M^{-1} A) = \setcond{M^{-1} z}{z \in \lattice(A)}$, which corresponds to a change of a coordinate system in $\R^m$ and ensures that $\gcd(A)=1$. 
			
			\emph{Sparsity bound \eqref{card:gamma:bound}.} The matrix $B$ gives rise to the lattice $\Lambda := \lattice(B)$ of rank $m$, while $\Lambda$ determines the finite Abelian group $\Z^m / \Lambda$. 
			
			Consider the canonical homomorphism $\phi : \Z^m \to \Z^m / \Lambda$, sending an element of $\Z^m$ to its coset modulo $\Lambda$.  Since $\gcd(A)=1$, we have $\lattice(A)= \Z^m$, which implies
			\(
				\gen{ T } = \Z^m / \Lambda
			\)
			for $T:= \{\phi(\vec a_{m+1}),\ldots, \phi(\vec a_n) \}$. For every non-redundant subset $S$ of $T$, we have 
			\begin{align*}
				|S| & \le \kappa( \Z^m / \Lambda) & & \text{(by Theorem~\ref{thm:max:card:nonredundant})}
				\\ & \le \Omega_m(|\det(A_\tau)|) & & \text{(by Lemmas~\ref{lem:kappa:for:sum:cyclic} and \ref{lem:group:from:lattice}).}
			\end{align*}
			Fixing a set $I \subseteq \{m+1,\ldots,n\}$ that satisfies $|I|=|S|$ and $S =  \setcond{\phi(\vec a_i)}{i \in I}$, we reformulate 
			$\gen{S} = \Z^m / \Lambda$ as 
			\(
				\Z^m = \lattice(A_I) + \Lambda = \lattice(A_I) + \lattice(A_\tau) = \lattice(A_{I \cup \tau}).
			\)
			Thus, \eqref{card:gamma:bound} holds for $\gamma = I \cup \tau$.
			
			\emph{Construction of $\gamma$ in polynomial time.} The matrix $M$ used in the reduction to the case $\gcd(A) = 1$ can be constructed in polynomial time: one can obtain $M$ from the Hermite Normal Form of $A$ (with respect to the column transformations) by discarding zero columns. For the determination of $\gamma$, the set $I$ that defines the non-redundant  subset $S = \setcond{\phi(\vec a_i)}{i \in I}$ of $\Z^m / \Lambda$ needs to be determined. Start with $I = \{m+1,\ldots,n\}$ and iteratively check if some of the elements $\phi(\vec a_i) \in \Z^m / \Lambda$, where $i \in I$, is in the group generated by the remaining elements. 
			Suppose $j \in I$ and we want to check if $\phi(\vec a_j)$ is in the group generated by all $\phi(\vec a_i)$ with $i \in I\setminus \{j\}$. Since $\Lambda = \lattice(A_\tau)$, this is equivalent to checking $\vec a_j \in \lattice(A_{I \setminus \{j\} \cup \tau})$ and is thus reduced to solving a system of linear Diophantine equations with the left-hand side matrix $A_{I \setminus \{j\} \cup \tau}$ and the right-hand side vector $\vec a_j$. Thus, carrying the above procedure for every $j \in I$ and removing $j$ from $I$ whenever $\vec a_j \in \lattice(A_{I \setminus \{j\} \cup \tau})$, we eventually arrive at a set $I$ that determines a non-redundant subset $S$ of $\Z^m / \Lambda$. This is done by solving at most $n-m$ linear Diophantine systems in total, where the matrix of each system is a sub-matrix of $A$ and the right-hand vector of the system is a column of $A$. 
\end{proof}
	
	\begin{remark}[Optimality of the bounds]
		For a given $\Delta \in \Z_{\ge 2}$ let us consider matrices $A \in \Z^{m \times n}$ 
		with $\Delta = |\det(A_\tau)| / \gcd(A)$. We construct a matrix $A$ that shows the optimality of the bound \eqref{card:gamma:bound}.
				As in the proof of Theorem~\ref{thm:sparsifying:lattice:A}, we assume $\tau =[m]$ and use the notation $B = A_\tau$.
		Consider the prime factorization $\Delta = p_1^{n_1} \cdots p_s^{n_s}$.  We will fix the matrix $B$ to be a diagonal matrix with diagonal entries $d_1,\ldots,d_m \in \Z_{>0}$ so that $\det(B) = d_1 \cdots d_m = \Delta$. 
		
		The diagonal entries are defined by distributing the prime factors of $\Delta$ among the diagonal entries of $B$. 
		If the multiplicity $n_i$ of the prime $p_i$ is less than $m$, we introduce $p_i$ as a factor of multiplicity $1$ in $n_i$ of the $m$ diagonal entries of $B$. If the multiplicity $n_i$ is at least $m$, we are able distribute the factors $p_i$ among \emph{all} of the diagonal entries of $B$ so that each diagonal entry contains the factor $p_i$ with multiplicity at least $1$. 
		
		The group $\Z^m / \Lambda = \Z^m / \lattice(B)$ is a direct sum of $m$ cyclic groups $G_1,\ldots, G_m$ of orders $d_1,\ldots,d_m$, respectively. By the Chinese Remainder Theorem, these cyclic groups can be further decomposed into the direct sum of primary cyclic groups. By our construction, the prime factor $p_i$ of the multiplicity $n_i < m$ generates a cyclic direct summand of order $p_i$ in $n_i$ of the subgroups $G_1,\ldots,G_m$. If $n_i \ge m$, then each of the groups $G_1,\ldots,G_m$ has a direct summand, which is a non-trivial cyclic group whose order is a power of $p_i$. Summarizing, we see that the decomposition of $\Z^m / \Lambda$ into primary cyclic groups contains $n_i$ summands of order $p_i$, when $n_i < m$, and $m$ summands, whose order is a power of $p_i$, when $n_i \ge m$. The total number of summands is thus $\sum_{i=1}^s \min \{m, n_i\} = \Omega_m(\Delta)$.
		
		Now, fix $n= m + \Omega_m(\Delta)$ and choose columns $\vec a_{m+1},\ldots,\vec a_n$ so that $\phi(\vec a_{m+1}),$ $\ldots, \phi(\vec a_n)$ generate all direct summands in the decomposition of $\Z^m / \Lambda$ into primary cyclic groups. With this choice, $\phi(\vec a_{m+1}),\ldots, \phi(\vec a_n)$ generate $\Z^m / \Lambda$, which means that $\lattice(A) = \Z^m$ and implies $\gcd(A)=1$. On the other hand, any proper subset $\{\phi(\vec a_{m+1}),\ldots, \phi(\vec a_n)\}$ generates a proper subgroup of $\Z^m / \Lambda$, as some of the direct summands in the decomposition of $\Z^m / \Lambda$ into primary cyclic groups will be missing. This means $\lattice(A_{[m] \cup I}) \varsubsetneq \Z^m$ for every $I \varsubsetneq \{m+1,\ldots,n\}$. 
\end{remark}

\begin{proof}[Proof of Corollary~\ref{cor:sparsity:dioph:sys}]
	Feasiblity of \eqref{lin:dioph:sys} can be expressed as $\vec b \in \lattice(A)$. Choose $\gamma$ from the assertion of Theorem~\ref{thm:sparsifying:lattice:A}. One has $\vec b \in \lattice(A) = \lattice(A_\gamma)$ and so there exists a solution $\vec x$ of \eqref{lin:dioph:sys} whose support is a subset of $\gamma$. This sparse solution $\ve x$ can be computed by solving the Diophantine system with the left-hand side matrix $A_\gamma$ and the right-hand side vector $\ve b$. 
\end{proof}

\begin{proof}[Proof of Corollary~\ref{cor:pos:spanning:columns}]
	Assume that the Diophantine system $A \vec x = \vec b, \  \vec x \in \Z^n$ has a solution. It suffices to show that, in this case, the integer-programming feasibility problem $A \vec x  = \vec b, \ \vec x \in \Z_{\ge 0}^n$ has a solution, too, and that one can find a solution of the desired sparsity to the integer-programming feasibility problem in polynomial time.
	
	One can determine $\gamma$ as in Theorem~\ref{card:gamma:bound} in polynomial time. Using $\gamma$, we can determine a solution $\vec x^\ast  = (x^\ast_1,\ldots,x^\ast_n)^\top \in \Z^n$ of the Diophantine system $A \vec x = b, \ \vec x \in \Z^n$ satisfying $x^\ast_i =0$ for $i \in [n] \setminus \gamma$ in polynomial time, as described in the proof of Corollary~\ref{cor:sparsity:dioph:sys}.
	
	Let $\vec a_1,\ldots,\vec a_n$ be the columns of $A$.  Since the matrix $A_\tau$ is non-singular, the $m$ vectors vectors $\vec a_i$, where $i \in \tau$, together with the vector $\vec v= -\sum_{i \in \tau} \vec a_i$ positively span $\R^n$. Since all columns of $A$ positive span $\R^n$, the conic version of the Carath\'eodory theorem implies the existence of a set $\beta \subseteq [m]$ with $|\beta| \le m$, such that $\vec v$ is in the conic hull of $\setcond{\vec a_i}{i \in \beta}$. Consequently, the set $\setcond{\vec a_i}{i \in \beta \cup \tau}$ and by this also the larger set $\setcond{\vec a_i}{i \in \beta \cup \gamma}$ positively span $\R^m$.  Let $I = \beta \cup \gamma$. By construction, $|I| \le |\beta| + |\gamma| \le m + |\gamma|$.
	
	Since the vectors $\vec a_i$ with $i \in I$ positively span $\R^m$, there exist a choice of rational coefficients $\lambda_i>0$ ($i \in I$) with $\sum_{i \in I} \lambda_i \vec a_i = 0$. After rescaling we can assume $\lambda_i \in \Z_{>0}$. Define $\vec x' = (x_1',\ldots,x_n')^\top \in \Z_{\ge 0}^n$ by setting $x'_i = \lambda_i$ for $i \in I$ and $x'_i=0$ otherwise. The vector $\vec x'$ is a solution of $A \vec x = \vec 0$. Choosing $N \in \Z_{>0}$ large enough, we can ensure that the vector $\ve x^\ast + N \ve x'$ has non-negative components. Hence, $\ve x = \ve x^\ast + N \ve x'$ is a solution of the system $A \vec x = \vec b, \ \vec x \in \Z_{\ge 0}^n$ satisfying the desired sparsity estimate. The coefficients $\lambda_i$ and the number $N$ can be computed in polynomial time.
\end{proof}
\begin{proof}[Proof of Corollary~\ref{cor:mixed:sign:knapsack}]
	The assertion follows by applying Corollary~\ref{cor:pos:spanning:columns} for $m=1$ and all $\tau = \{i\}$ with $i \in [n]$.
\end{proof}

\section{Proof of Theorem~\ref{thm:positive:knapsack}}


\begin{lemma}
\label{basic:lemma}
Let $a_1,\ldots,a_t \in \Z_{>0}$, where $t \in \Z_{> 0}$. If $t > 1 + \log_2 (a_1)$, then the system
\begin{align*}
& y_1 a_1 + \cdots + y_t a_t = 0,
\\  & y_1 \in \Z_{\ge 0}, \  y_2,\ldots,y_t \in \{-1,0,1\}.
\end{align*}
in the unknowns $y_1,\ldots,y_t$ has a solution that is not identically equal to zero.
\end{lemma}
\begin{proof}
	The proof is inspired by the approach in \cite[\S~3.1]{averkov2012:SIDMA} (used in a different context) that suggests to reformulate the underlying equation over integers as two strict inequalities and then use Minkowski's first theorem \cite[Ch.~VII, Sect.~3]{MR1940576} from the geometry of numbers. 
	Consider the convex set $Y \subseteq \R^t$ defined by $2t$ strict linear inequalities
	\begin{align*}
	-1  < & y_1 a_1 + \cdots + y_t a_t < 1,
	\\ -2  < & y_i < 2 \text{ for all } i \in \{2,\ldots,t\}.
	\end{align*}
	Clearly, the set $Y$ is the interior of a hyper-parallelepiped and can also be described as $Y = \setcond{\ve y \in \R^t}{\|M \vec y\|_\infty < 1}$, where $M$ is the upper triangular matrix
	\[
	M = \begin{pmatrix} a_1 & a_2 & \cdots & a_t
	\\ & 1/2 & &
	\\ & & \ddots &
	\\ & & & 1/2
	\end{pmatrix}.
	\]
	It is easy to see that the $t$-dimensional volume $\vol(Y)$ of $Y$ is  \[ \vol(Y)=\vol(M^{-1}[-1,1]^t)=\frac{1}{\det(M)}2^t=\frac{4^t}{2 a_1}.\]
	The assumption $t > 1 + \log_2 (a_1)$ implies that the volume of $Y$ is strictly larger than $2^t$. Thus, by Minkowski's first theorem, the set $Y$ contains a non-zero integer vector $\ve y = (y_1,\ldots,y_t)^\top \in \Z^t$. Without loss of generality we can assume that $y_1 \ge 0$ (if the latter is not true, one can replace $\ve y$ by $-\ve y$). The vector $\ve y$ is a desired solution from the assertion of the lemma.
	\end{proof}

\begin{proof}[Proof of Theorem~\ref{thm:positive:knapsack}]
	Without loss of generality we can assume that $\gcd(\ve a)=1$. In fact, if $b$ is divisible by $\gcd(\ve a)$ we can convert $\ve a^\top \ve x = b$ to $\overline{\ve a}^\top \ve x = \overline{b}$ with $\overline{\ve a} = \frac{\ve a}{\gcd(\ve a)}$ and $\overline{b} = \frac{b}{\gcd(\ve a)}$, and, if $b$ is not divisible by $\gcd(\ve a)$, the knapsack feasibility problem $\ve a^\top \ve x = b, \ \ve x \in \Z_{\ge 0}^n$ has no solution.
	
	Without loss of generality, let $ a_1 = \min \{a_1,\ldots,a_n\}$. We need to show the existence of solution of the knapsack feasibility problem satisfying $\|\ve x\|_0 \le 1 + \log_2( a_1)$.
	
	Choose a solution $\ve x=(x_1,\ldots,x_n)^\top$ of the knapsack feasibility problem with the property that the number of indices $i \in \{2,\ldots,n\}$ for which $x_i \ne 0$ is minimized. Without loss of generality we can assume that, for some $t \in \{2,\ldots,n\}$ one has $x_2 > 0,\ldots, x_t > 0, x_{t+1} = \cdots = x_n = 0$.  Lemma~\ref{basic:lemma} implies $t \le 1 + \log_2 (a_1) $. In fact, if the latter was not true, then a solution $\ve y \in \R^t$ of the system in Lemma~\ref{basic:lemma} could be extended to a solution $\ve y \in \R^n$ by appending zero components. It is clear that some of the components $y_2,\ldots,y_t$ are negative, because $a_2 > 0,\ldots, a_t > 0$. It then turns out that, for an appropriate choice of $k \in \Z_{\ge 0}$, the vector $\ve x'= (x_1',\ldots,x'_n)^\top = \ve x + k \ve y $ is a solution of the same knapsack feasibility problem satisfying $x_1' \ge 0,\ldots, x'_t \ge 0, \ x'_{t+1} = \cdots = x'_n = 0$ and $x'_i=0$ for at least one $i \in \{2,\ldots,t\}$. Indeed, one can choose $k$ to be the minimum among all $a_i$ with $i \in \{2,\ldots,t\}$ and $y_i=-1$.
	
	The existence of $\ve x'$ with at most $t-1$ non-zero components $x_i'$ with $i \in \{2,\ldots,n\}$ contradicts the choice of $\ve x$ and yields the assertion.
\end{proof}

\vskip .3cm
\noindent {\bf Acknowledgements} The second author is supported  by the  DFG (German Research Foundation) within the project number 413995221.
 The third author acknowledges partial support from the NSF grants 1818969 and CCF-1934568.

\bibliographystyle{plain}
\bibliography{bib.bib}

\begin{thebibliography}{10}

\bibitem{Support}
I.~Aliev, J.~A. De~Loera, F.~Eisenbrand, T.~Oertel, and R.~Weismantel.
\newblock The support of integer optimal solutions.
\newblock {\em SIAM J. Optim.}, 28(3):2152--2157, 2018.

\bibitem{ADOO}
I.~Aliev, J.~A. De~Loera, T.~Oertel, and C.~O'Neill.
\newblock Sparse solutions of linear diophantine equations.
\newblock {\em SIAM Journal on Applied Algebra and Geometry}, 1(1):239--253,
  2017.

\bibitem{averkov2012:SIDMA}
G.~Averkov.
\newblock On the size of lattice simplices with a single interior lattice
  point.
\newblock {\em SIAM Journal on Discrete Mathematics}, 26(2):515--526, 2012.

\bibitem{MR1940576}
A.~I. Barvinok.
\newblock {\em A {C}ourse in {C}onvexity}, volume~54 of {\em Graduate Studies
  in Mathematics}.
\newblock American Mathematical Society, Providence, RI, 2002.

\bibitem{barvinok-woods-2003}
A.~I. Barvinok and K.~Woods.
\newblock Short rational generating functions for lattice point problems.
\newblock {\em Journal of the AMS}, 16(4):957--979, 2003.

\bibitem{CS_survey}
H.~Boche, R.~Calderbank, G.~Kutyniok, and J.~Vyb\'{\i}ral.
\newblock A survey of compressed sensing.
\newblock In {\em Compressed sensing and its applications}, Appl. Numer.
  Harmon. Anal., pages 1--39. Birkh\"{a}user/Springer, Cham, 2015.

\bibitem{ECC}
E.~Cand\`es, M.~Rudelson, T~Tao, and R.~Vershynin.
\newblock Error correction via linear programming.
\newblock {\em 46th Annual IEEE Symposium on Foundations of Computer Science
  (FOCS'05)}, pages 668--681, 2005.

\bibitem{candesetal2006stable}
E.~J. Cand\`es, J.~K. Romberg, and T.~Tao.
\newblock Stable signal recovery from incomplete and inaccurate measurements.
\newblock {\em Comm. Pure Appl. Math.}, 59(8):1207--1223, 2006.

\bibitem{candestao}
E.~J. Cand\`es and T.~Tao.
\newblock Decoding by linear programming.
\newblock {\em IEEE Trans. Inform. Theory}, 51(12):4203--4215, 2005.

\bibitem{EisenbrandShmonin2006}
F.~Eisenbrand and G.~Shmonin.
\newblock Carath\'eodory bounds for integer cones.
\newblock {\em Oper. Res. Lett.}, 34(5):564--568, 2006.

\bibitem{PROMP}
A.~Flinth and G.~Kutyniok.
\newblock P{ROMP}: {A} sparse recovery approach to lattice-valued signals.
\newblock {\em Appl. Comput. Harmon. Anal.}, 45(3):668--708, 2018.

\bibitem{Lenya}
L.~Fukshansky, D.~Needell, and B.~Sudakov.
\newblock An algebraic perspective on integer sparse recovery.
\newblock {\em Appl. Math. Comput.}, 340:31--42, 2019.

\bibitem{Geroldinger+HalterKoch}
A.~Geroldinger and F.~Halter-Koch.
\newblock {\em Non-Unique Factorizations: Algebraic, Combinatorial and Analytic
  Theory}.
\newblock Pure and Applied Mathematics. Chapman and Hall/CRC, 2006.

\bibitem{haase2018survival}
C.~Haase.
\newblock A survival guide to {P}resburger arithmetic.
\newblock {\em ACM SIGLOG News}, 5(3):67--82, 2018.

\bibitem{hardy2008introduction}
G.H. Hardy, E.M. Wright, R.~Heath-Brown, and J.~Silverman.
\newblock {\em An {I}ntroduction to the {T}heory of {N}umbers}.
\newblock Oxford mathematics. OUP Oxford, 2008.

\bibitem{kannan1992lattice}
Ravi Kannan.
\newblock Lattice translates of a polytope and the frobenius problem.
\newblock {\em Combinatorica}, 12(2):161--177, 1992.

\bibitem{Konyagin}
S.~V. Konyagin.
\newblock On the recovery of an integer vector from linear measurements.
\newblock {\em Mat. Zametki}, 104(6):863--871, 2018.

\bibitem{LOVASZ1987187}
L.~Lov\'asz.
\newblock Matching structure and the matching lattice.
\newblock {\em Journal of Combinatorial Theory, Series B}, 43(2):187 -- 222,
  1987.

\bibitem{natarajan}
B.~K. Natarajan.
\newblock Sparse approximate solutions to linear systems.
\newblock {\em SIAM J. Comput.}, 24(2):227--234, 1995.

\bibitem{MR3950898}
T.~Oertel, J.~Paat, and R.~Weismantel.
\newblock Sparsity of integer solutions in the average case.
\newblock In {\em Integer programming and combinatorial optimization}, volume
  11480 of {\em Lecture Notes in Comput. Sci.}, pages 341--353. Springer, Cham,
  2019.

\bibitem{ramirez1996complexity}
J.~L. Ram{\'\i}rez-Alfons{\'\i}n.
\newblock Complexity of the {F}robenius problem.
\newblock {\em Combinatorica}, 16(1):143--147, 1996.

\bibitem{MIMO}
M.~Rossi, A.~M. Haimovich, and Y.~C. Eldar.
\newblock Spatial compressive sensing for {MIMO} radar.
\newblock {\em IEEE Trans. Signal Process.}, 62(2):419--430, 2014.

\bibitem{Schrijver}
A.~Schrijver.
\newblock {\em Theory of {L}inear and {I}nteger {P}rogramming}.
\newblock Wiley-Interscience Series in Discrete Mathematics. John Wiley \&
  Sons, Ltd., Chichester, 1986.
\newblock A Wiley-Interscience Publication.

\bibitem{SeboCaratheodory}
A.~Seb\"{o}.
\newblock Hilbert bases, {C}arath\'eodory's {T}heorem and combinatorial
  optimization.
\newblock In {\em Proceedings of the 1st Integer Programming and Combinatorial
  Optimization Conference}, pages 431--455, Waterloo, Ont., Canada, Canada,
  1990. University of Waterloo Press.

\end{thebibliography}

\end{document}